\documentclass[a4,11pt]{article}
\usepackage{amsthm}

\usepackage[a4paper]{geometry}
\usepackage[utf8x]{inputenc}
\usepackage{amsfonts}
\usepackage{amsmath}
\usepackage{amssymb}
\usepackage[mathscr]{eucal}
\usepackage{color}
\usepackage[colorlinks=true,citecolor=black,linkcolor=black,urlcolor=blue]{hyperref}
\usepackage{comment}
\usepackage[symbol]{footmisc}

\usepackage{cite}
\usepackage{cases}
\usepackage{comment}
\usepackage{paralist}

\usepackage[dvipsnames]{xcolor}

\newcommand{\noop}[1]{}

\def\mL{\mathsf{L}}

\def\cE{\mathcal E}

\def\cH{\mathcal H}

\def\cO{\mathcal O}
\def\cP{\mathcal P}
\def\cQ{\mathcal Q}

\def\PG{{\rm PG}}

\def\PSL{{\rm PSL}}

\def\d{\downarrow}

\allowdisplaybreaks

\newcommand{\newnumbered}[2]{\newtheorem{#1}[theorem]{#2}}
\newcommand{\newunnumbered}[2]{\newtheorem{#1}[theorem]{#2}}
\newcommand{\Title}[2]{\title{#1}\newcommand{\Acknowledgements}{\section*{Acknowledgements} #2}}
\newcommand{\Author}[2][]{\author{#2}}
\newcommand{\Comma}{\and}
\newcommand{\Und}{\and}
\newcommand{\br}{, }
\newcommand{\fs}{. }
\newcommand{\thanksone}[3][]{#1\thanks{#3\email{\tt #2}}}
\newcommand{\thankstwo}[3][]{#1\thanks{#3\email{\tt #2}}}
\newcommand{\thanksthree}[3][]{#1\thanks{#3\email{\tt #2}}}
\newcommand{\email}[1]{#1}
\newcommand{\classno}[2][2000]{}

\newtheorem{theorem}{Theorem}[section] 
\newtheorem{lemma}[theorem]{Lemma}     
\newtheorem{corollary}[theorem]{Corollary}

\newtheorem{claim}[theorem]{Claim}
\newnumbered{assertion}{Assertion}    
\newnumbered{conjecture}{Conjecture}  
\newnumbered{definition}{Definition}
\newnumbered{hypothesis}{Hypothesis}
\newnumbered{remark}{Remark}
\newnumbered{note}{Note}
\newnumbered{observation}{Observation}
\newnumbered{problem}{Problem}
\newnumbered{question}{Question}
\newnumbered{algorithm}{Algorithm}
\newnumbered{example}{Example}
\newunnumbered{notation}{Notation} 
\numberwithin{equation}{section}

\renewcommand{\labelenumi}{\theenumi}

\Title{A modular equality for $m$--ovoids of elliptic quadrics}{The research of Alexander Gavrilyuk is partially supported by JSPS KAKENHI Grant Number
18K03395. The research of Francesco Pavese is supported by the Italian National Group for Algebraic and Geometric Structures and their Applications (GNSAGA - INdAM).}

\Author[Alexander L. Gavrilyuk, Klaus Metsch, Francesco Pavese]{%
\thanksone[Alexander L. Gavrilyuk]{gavrilyuk@riko.shimane-u.ac.jp}{%
Interdisciplinary Faculty of 
Science and Engineering\br
Shimane University\br
Matsue, Japan\fs
}
\Comma
\thanksthree[Klaus Metsch]{klaus.metsch@math.uni-giessen.de
}{%
Mathematisches Institut\br 
Justus-Liebig-Universit\"{a}t\br 
Gie\ss en, Germany\fs
}
\Und
\thankstwo[Francesco Pavese]{francesco.pavese@poliba.it}{%
Dipartimento di Meccanica\br 
Matematica e Management\br
Politecnico di Bari\br
Bari, Italy\fs
}
}

\classno{51E20 (primary), 05B25 (secondary), 51A50}

\date{\today}

\begin{document}

\maketitle

\begin{abstract}
An $m$-ovoid of a finite polar space $\mathcal{P}$ is a set $\mathcal{O}$ of points such that every maximal subspace of $\mathcal{P}$ contains exactly $m$ points of $\mathcal{O}$. In the case when $\mathcal{P}$ is an elliptic quadric $\cQ^-(2r+1, q)$ of rank $r$ in $\mathbb{F}_q^{2r+2}$, we prove that an $m$-ovoid exists only if $m$ satisfies a certain modular equality, which depends on $q$ and $r$. This condition rules out many of the possible values of $m$. Previously, only a lower bound on $m$ was known, which we slightly improve as a byproduct of our method. We also obtain a characterization of the $m$-ovoids of $\cQ^{-}(7,q)$ for $q = 2$ and $(m, q) = (4, 3)$.
\end{abstract}

\section{Introduction}

Let $\Gamma$ be a finite connected regular graph on $v$ vertices and of valency $k$, $Y$ a proper subset of the vertex set of $\Gamma$. If $\theta^+$ and $\theta^-$ denote the second largest and least eigenvalues of $\Gamma$ respectively, then the number $N$ of ordered pairs of adjacent vertices of $Y$ satisfies (see \cite[Theorem 2.1]{Eisfeld}, \cite[Proposition 3.8]{DeBruynSuzuki})
\begin{equation}\label{eq:Y}
\theta^-|Y|+\frac{k-\theta^-}{v}|Y|^2\leq 
N\leq
\theta^+|Y|+\frac{k-\theta^+}{v}|Y|^2.
\end{equation}
The case of equality in Eq. \eqref{eq:Y} often gives rise to interesting combinatorial objects; 
in particular, when $\Gamma$ is related to 
incidence structures in finite geometry.

Let $\PG(n, q)$ denote the projective space of dimension $n$ with underlying vector space $V:=\mathbb{F}_q^{n+1}$ over the finite field $\mathbb{F}_q$ with $q$ elements. For a non-degenerate quadratic (or reflexive sesquilinear) form $f$ on $V$, the {\bf classical polar space} $\mathcal{P}$ associated with $f$ is the incidence structure formed by the totally singular (or totally isotropic, respectively) subspaces with respect to $f$; their incidence is defined by symmetrized containment \cite{pps,HT}. We consider the elements of $\mathcal{P}$ as subspaces of $\PG(n,q)$, so they are projective points, lines, etc. A maximal subspace of $\mathcal{P}$ has dimension $r-1$, where $r$ is the Witt index of $f$, also called the {\bf rank} of $\mathcal{P}$; such a subspace is called a {\bf generator}. (For further details, the reader is referred to Section \ref{sect:prelim}.)

A set $\cO$ of points of $\cP$ is an {\bf $m$-ovoid} if every generator of $\cP$ meets $\cO$ in exactly $m$ points. Equivalently, if $\Gamma$ is the collinearity graph of a finite polar space $\mathcal{P}$, then a set $Y$ that attains equality in the left-hand inequality of Eq. \eqref{eq:Y} is an $m$-ovoid of $\mathcal{P}$, for some natural $m$ (see \cite{BKLP, DeBruynSuzuki}). The notion of $m$-ovoids, which goes back to a classical work of B. Segre \cite{S}, was  introduced by J. Thas for generalized quadrangles in \cite{Thas}, and  extended to finite polar spaces of higher rank in \cite{ST}. A set of points of $\cP$ is called {\bf tight} \cite{Drudge} if it attains  equality in the right-hand side of Eq. \eqref{eq:Y}. 
In \cite{BKLP} a uniform algebraic framework for ovoids and tight sets was developed, and their connections with various geometric objects were explored.
\smallskip

The central problem concerning $m$-ovoids in a polar space $\cP$ is to determine the values of $m$, for which $\cP$ possesses an $m$-ovoid.
In this paper, we focus on the elliptic polar spaces (quadrics) $\cQ^-(2r+1, q)$ of rank $r$, which arise from a nondegenerate orthogonal form of Witt index $r$ in the $(2r+2)$-dimensional vector space $V$ over $\mathbb{F}_q$. An $m$-ovoid of $\cQ^{-}(2r+1,q)$ has $m(q^{r+1}+1)$ points, and 
it is said to be {\bf trivial} if it is empty ($m=0$) or consists of all points of the space ($m = \frac{q^r-1}{q-1}$). As the complement of an $m$-ovoid is an $\left(\frac{q^r-1}{q-1} - m\right)$-ovoid, one may assume $m\leq \frac{q^r-1}{2(q-1)}$.

In 1965 Segre \cite[p. 162]{S} proved that if an elliptic quadric $\cQ^-(5, q)$ of $\PG(5, q)$, $q$ odd, has an $m$-ovoid, then $m = \frac{q + 1}{2}$, and he called such a $\left(\frac{q+1}{2}\right)$-ovoid a {\bf hemisystem}. He also constructed a hemisystem of $\cQ^-(5, 3)$, admitting a group isomorphic to $\PSL(3, 4)$. In \cite{BruenHir}, by extending Segre's result, it was shown that $\cQ^-(5, q)$, $q$ even, possesses no non-trivial $m$-ovoids, see also \cite[Section 19.3]{H2}. Several constructions of hemisystems of $\cQ^-(5, q)$ have been presented in the literature in the last fifteen years \cite{BGR, BGR1, BLMX, CPa, CPe, KNS}. However, not much is known about $m$-ovoids of $\cQ^-(2r+1, q)$ for $r > 2$. It was shown in \cite[Theorem 13]{BKLP} that if an $m$-ovoid of $\cQ^-(2r+1, q)$ exists, then $m \ge (\sqrt{4 q^{r+1} + 9} - 3)/(2q-2)$ (see also  Remark \ref{rm:1} below).
\smallskip

Here we obtain the following (non-existence) result for $m$-ovoids of elliptic quadrics.
\begin{theorem}\label{th:1}
If $\cQ^-(2r+1,q)$ possesses an $m$-ovoid, then
\begin{equation}\label{eq:main}
F(m) \equiv 0 \pmod{q+1}, 
\end{equation}
where 
\[    F(m)=\begin{cases} 
    m^2 - m & \mbox{~if~}r\mbox{~is odd,} \\ 
    m^2 & \mbox{~if~}r\mbox{~is even and~}q\mbox{~is even,} \\ 
    m^2 + \frac{q+1}{2} m & \mbox{~if~}r\mbox{~is even and~}q\mbox{~is odd.}  
    \end{cases}
\]
\end{theorem}

An analysis of Eq. \eqref{eq:main} 
in Section \ref{sect:eq} shows that the admissible  values of $m$ for an $m$-ovoid of $\cQ^-(2r+1, q)$ are asymptotically rare.

Note that certain properties of $m$-ovoids of elliptic quadrics mirror those of tight sets of hyperbolic quadrics $\cQ^{+}(2r+1,q)$; in fact, a result of a similar spirit as Theorem \ref{th:1} was shown for tight sets of hyperbolic quadrics \cite{GMo, GMe, G}. Tight sets, the counterpart of $m$-ovoids, have been studied intensely in recent years (perhaps, with the main focus on those in the hyperbolic quadric $\cQ^{+}(5,q)$, also known as the {\em Cameron-Liebler line classes} in $\PG(3,q)$, 
see \cite{Feng} and references therein).

The proof of Theorem \ref{th:1} occupies Sections \ref{sect:prelim} and \ref{sect:eq}, 
and it is based on the following approach.
Fix a maximal flag $P_0\subset \ell_0\subset \pi_0 \subset 
\ldots \subset \Pi_0$ in $\cQ^-(2r+1, q)$ and define a sequence of quotient polar spaces: $\cQ_{r - 1}$ in $P_0^{\perp}/P_0$, $\cQ_{r-2}$ in $\ell_0^{\perp}/\ell_0$, etc., induced by $\cQ_r:=\cQ^-(2r+1, q)$. {Suppose that  $\mathcal{O}$ is an $m$-ovoid of $\cQ_r$, and 
$\mu_0\colon \cQ_r \to \mathbb{Z}$ is 
the characteristic function of $\cO$. 
In Section \ref{sect:prelim}, we show that $\mu_0$ induces a function $\mu_1\colon \cQ_{r-1} \to \mathbb{Z}$, called a {\em weighted ovoid} of $\cQ_{r-1}$, 
which in some sense generalizes the notion of an $m$-ovoid. Furthermore, such a function $\mu_i$ induces a weighted ovoid $\mu_{i+1}\colon \cQ_{r - i - 1} \to \mathbb{Z}$, for every $i=1,\ldots,r-2$. 
Put $\|\mu_i\|^2:=\sum_{P\in \cQ_{r-i}}\big(\mu_i(P)\big)^2$, and note that $\|\mu_0\|^2$ is simply equal to $|\mathcal{O}|=m(q^{r+1}+1)$. We prove that, for $i=0,\ldots,r-2$, $\|\mu_i\|^2$ can be expressed via $\|\mu_{i+1}\|^2$. Arguing by induction on $i$ for $i=1,\ldots,r-2$ implies that $\|\mu_1\|^2\equiv G(q,r,m) \pmod{q+1}$, where $G$ is a certain function. On the other hand, as $\|\mu_0\|^2=m(q^{r+1}+1)$, 
applying the induction step with $i=0$ shows that $\|\mu_1\|^2$ is another function in $q,r,m$, say $\|\mu_1\|^2=E(q,r,m)$. 
Thus, $E(q,r,m)\equiv G(q,r,m)\pmod{q+1}$ should have an integer solution in $m$, which, as shown in Section \ref{sect:eq}, gives the conclusion of Theorem \ref{th:1}.}

Finally, in Section \ref{sect:ineq}, by using the technique developed in Section \ref{sect:prelim}, we  slightly improve the above-mentioned lower bound for $m$ from \cite[Theorem 13]{BKLP}. We also provide a complete classification of the $m$-ovoids of $\cQ^-(7, 2)$ and a characterization of the $4$-ovoids of $\cQ^-(7, 3)$.

\begin{remark}\label{rm:1}
In \textup{\cite{K}} it is shown that the so-called {\em field reduction} allows one to construct:
\begin{compactitem}
   \item an $\left(m\frac{q^e - 1}{q - 1}\right)$-ovoid of $\cQ^-(2e(r+1) - 1, q)$ from an $m$-ovoid of $\cQ^-(2r+1, q^e)$,
    \item an $\left(m\frac{q^{2e} - 1}{q - 1}\right)$-ovoid of $\cQ^-(2e(2r+1) - 1, q)$ from an $m$-ovoid of a Hermitian variety $\cH(2r, q^{2e})$.
\end{compactitem}

However, apart from a $\left(\frac{q^{(4r+2)/3} - 1}{q^2-1}\right)$-ovoid of $\cH(2r, q^2)$, $r\equiv 1\mod{3}$, see \textup{\cite[Corollary 7.39]{HT}}, 
which in turn is obtained by the field reduction
from a 1-ovoid of $\cH(2,q^{(4r+2)/3})$, 
no non-trivial $m$-ovoids of $\cH(2r, q^2)$ are known to exist. Thus, the only examples of non-trivial $m$-ovoids of elliptic quadrics of rank at least $3$ arise by applying the field reduction to all points of $\cQ^-(2r+1, q^e)$, $r\geq 1, e\geq 2$, or to a hemisystem of $\cQ^-(5, q^e)$, $e\geq 2$, or to all points of $\cH(2r,q^{2e})$, $r\geq $, $e\geq 2$.
\end{remark}

\section{Preliminary results}\label{sect:prelim}

In this section, we prepare technical results needed for the proof of Theorem \ref{th:1}. First, we recall some basic properties of elliptic polar spaces; further details can be found in \cite{pps, HT}. 

Let us consider an elliptic polar space (quadric) $\cQ_r :=\cQ^-(2r+1,q)$ of rank $r\geq 1$, formed by the set of projective points of $\PG(2r+1,q)$ satisfying $f({\bf x})=0$, where
\begin{equation*}
f({\bf x}):=f(x_0,\ldots,x_{2r+1})=x_0x_1 + \cdots + x_{2r-2}x_{2r-1} +g(x_{2r},x_{2r+1}),\quad {\bf x} \in \mathbb{F}_q^{2r+2},
\end{equation*}
and $g$ is a homogeneous irreducible polynomial of degree 2 over $\mathbb{F}_q$. The number of points in $\cQ_r$ is
\begin{equation*}
k_r := \frac{(q^r - 1)(q^{r+1} + 1)}{q - 1}.
\end{equation*}

The associated bilinear form $B({\bf x},{\bf y}):=f({\bf x}+{\bf y})-f({\bf x})-f({\bf y})$ defines the {\bf polarity} $\perp$ of $\PG(2r+1,q)$. 
Two points $X, Y$ of $\PG(2r+1, q)$ represented by  vectors ${\bf x}, {\bf y}$ are said to be {\bf orthogonal} if $B({\bf x},{\bf y})=0$. Moreover, two orthogonal points $X, Y\in \cQ_r$ either coincide, $X=Y$, or are {\bf collinear}, which means that the projective line joining $X,Y$ is entirely contained in $\cQ_r$. 

For a point $P$, denote by $P^{\perp}$ the set of points of $\PG(2r+1, q)$ orthogonal with $P$; such a set is a hyperplane of $\PG(2r+1, q)$ that is either {\bf tangent} or not according as $P$ belongs to $\cQ_r$ or not. Note that $P\in P^{\perp}$ if and only if $P\in \cQ_r$, and, for a point set (or a subspace) $\Pi$, let $\Pi^{\perp}$ denote $\cap_{P\in \Pi}P^{\perp}$. We use the term ($j$-){\bf space} to denote a ($j$-dimensional) projective subspace of the ambient projective space. Recall that the {\bf quotient} space $\PG(2r+1,q)/\Pi\cong \PG(2r-j,q)$ is a projective space whose points, lines, etc. are the subspaces of $\PG(2r+1,q)$ of dimension $j+1$, $j+2$, etc., containing $\Pi$.

Furthermore, when a point $P$ belongs to $\cQ_r$, $P^{\perp}$ is a tangent hyperplane $\cong\PG(2r,q)$ and $P^{\perp}/P$ is the quotient projective space $\cong\PG(2r-1,q)$. Then the subspaces of $\cQ_r$ of dimension $1$, $2$, etc., containing $P$, induce the {\bf quotient} polar space in $P^{\perp}/P$, which is projectively equivalent to an elliptic polar space $\cQ_{r-1}$ of rank $r- 1$. For the sake of simplicity, we will simply denote this polar space by $\cQ_{r-1}$. In particular, the set $\mathsf{L}(P)$ of lines of $\cQ_r$ through $P$ can be identified with the point set of $\cQ_{r-1}$ induced in $P^{\perp}/P$ by $\cQ_r$.
\smallskip

Tight sets and $m$-ovoids share the property that they exhibit precisely two intersection numbers with respect to tangent hyperplanes (see \cite{BKLP}). Namely, if $\cO$ is an $m$-ovoid of $\cQ_r$, then for every point $P\in \cQ_r$ we have
\begin{equation}\label{eq:ovoid-point}
|P^\perp \cap \cO| = 
\begin{cases}
(m-1)(q^r+1)+1 & \mbox{ if } P \in \cO, \\
m(q^r+1) & \mbox{ if } P \in \cQ_r \setminus \cO.
\end{cases}
\end{equation} 

Let $\mu \colon \PG(2r+1, q) \to \mathbb{Z}$ be a function defined on the points of $\PG(2r+1, q)$ such that $\mu(P) = 0$ if $P \in \PG(2r+1, q) \setminus \cQ_r$. For every subset $X$ of the point set of $\PG(2r+1, q)$, we define $\mu(X)=\sum_{P \in X} \mu(P)$; in particular, $\mu(X)=\mu(X\cap \cQ_r)$. Such a map $\mu$ is said to be a {\bf weighted ($m$-)ovoid} of $\cQ_r$, for some natural $m$, if the following property is satisfied:
\begin{equation}\label{eq:star}
\mu(P^\perp) + q^r \mu(P) = m(q^r+1) \mbox{ for every point } P \in \cQ_r, 
\end{equation}
and it immediately follows from Eq. \eqref{eq:ovoid-point} that the $(0,1)$-characteristic function of an $m$-ovoid of $\cQ_r$ is a weighted $m$-ovoid. 
Part (c) of the next lemma shows that Eq.  \eqref{eq:star} generalizes to arbitrary subspaces of the ambient projective space, 
and this fact will be intensively used in the proof of our main result.

\begin{lemma} \label{basic}
Let $\mu$ be a weighted $m$-ovoid of $\cQ_r$.
\begin{enumerate}
\renewcommand{\labelenumi}{\rm(\alph{enumi})}
\item $\mu(\cQ_r) = m(q^{r+1}+1)$. 
\item $\mu(H)=m(q^r+1)$ for every non-tangent hyperplane $H$.
\item $\mu(\Pi^\perp) + q^{r-j} \mu(\Pi) = m(q^{r-j}+1)$ for every $j$-space $\Pi$ of $\PG(2r+1,q)$.
\end{enumerate}
\end{lemma}
\begin{proof}
(a) Computing in two ways the sum of $\mu(P_2)$ over all pairs $(P_1,P_2)$, where $P_1, P_2$ are (not necessarily distinct) points of $\cQ_r$ and $P_2 \in P_1^\perp$, we obtain
\begin{eqnarray*}
\sum_{P_1 \in \cQ_r} \mu(P_1^\perp)&=&
\sum_{P_2 \in \cQ_r} \mu(P_2) (1+q k_{r-1})\\
&=& \mu(\cQ_r)(1+q k_{r-1}).
\end{eqnarray*}
By Eq. \eqref{eq:star}, the left-hand side equals $m(q^r+1) k_r - q^r \sum_{P_1 \in \cQ_r} \mu(P_1)=m(q^r+1) k_r - q^r \mu(\cQ_r)$. Simplifying and using that $(q^r + q k_{r-1} + 1)(q^{r+1}+1) = k_r (q^{r}+1)$, we get the result.

(b) Let $H$ be a non-tangent hyperplane of $\PG(2r+1, q)$. Computing in two ways the sum of $\mu(P_2)$ over all pairs $(P_1,P_2)$, where $P_1 \in H \cap \cQ_r$, $P_2 \in \cQ_r$ and $P_2 \in P_1^\perp$, we obtain
\begin{align*}
\sum_{P_1 \in H \cap \cQ_r} \mu(P_1^\perp)=
\sum_{P_2 \in \cQ_r} \mu(P_2) |H \cap \cQ_r \cap P_2^\perp|.
\end{align*}
By Eq. \eqref{eq:star}, the left-hand side equals $|H \cap \cQ_r| m(q^r+1)-\mu(H)q^r$. As $|H \cap \cQ_r| = \frac{q^{2r}-1}{q-1}$ and $|H \cap \cQ_r \cap P_2^\perp|$ equals $k_{r-1}$ if $P_2 \in \cQ_r \setminus H$ or $\frac{q^{2r-1}-1}{q-1}$ if $P_2 \in H \cap \cQ_r$, we have that
\begin{align*}
\left( q^r + \frac{q^{2r-1}-1}{q-1} - k_{r-1} \right) \mu(H) = m \left( (q^r+1) \frac{q^{2r}-1}{q-1} - (q^{r+1}+1) k_{r-1} \right),
\end{align*}
whence the result follows.

(c) Let $\Pi$ be a $j$-space of $\PG(2r+1, q)$. Consider the $\frac{q^{2r-j+1}-1}{q-1}$ hyperplanes $R^\perp$ with $R \in \Pi^\perp$. Every point of $\Pi$ lies in all these hyperplanes and every other point lies in $\frac{q^{2r-j}-1}{q-1}$ of these hyperplanes. Computing in two ways the sum of $\mu(P)$ over all pairs $(R^\perp, P)$, where $R \in \Pi^\perp$, $P \in R^\perp\cap \cQ_r$, shows that
\begin{align*}
\sum_{R \in \Pi^\perp} \mu(R^\perp) = \frac{q^{2r-j+1}-1}{q-1} \mu(\Pi) + \frac{q^{2r-j}-1}{q-1} \mu( \cQ_r \setminus \Pi).
\end{align*}
From (b) and the fact that $\mu$ is a weighted $m$-ovoid, the left-hand side is equal to 
\begin{align*}
m(q^r+1) |\Pi^\perp| - q^r \sum_{R \in \Pi^\perp \cap \cQ_r} \mu(R) = m(q^r+1) \frac{q^{2r-j+1}-1}{q-1} - q^r \mu(\Pi^\perp). 
\end{align*}
On the other hand, the right-hand side is equal to 
\begin{align*}
q^{2r-j} \mu(\Pi) + \frac{q^{2r-j}-1}{q-1} \mu(\cQ_r). 
\end{align*}
By (a), the assertion follows.
\end{proof}


Note that if $\cO$ is an $m$-ovoid of $\cQ_r$ and $\Pi_j$ is a $j$-space of $\PG(2r+1,q)$,
then it immediately follows from Lemma \ref{basic}(c) that 
\[ 
|\Pi_j^\perp \cap \cO| + q^{r-j}|\Pi_j \cap \cO|= m (q^{r-j}+1),\]
and this result is a counterpart of \cite[Lemma~2.1]{BM}.
\smallskip

For a point $P_0\in \cQ_r$, consider the quadric $\cQ_{r-1}$ induced in the projective space   $P_0^{\perp}/P_0\cong\PG(2r-1,q)$ by $\cQ_r$. Given a weighted ovoid $\mu$ of $\cQ_r$, define a function $\mu^{\downarrow}_{P_0} \colon \PG(2r-1,q) \to \mathbb{Z}$ by 
\begin{equation}\label{eq:hatmu}
\mu^{\downarrow}_{P_0}(\ell) = 
\begin{cases}
0 & \mbox{ if } \ell \notin \mL(P_0), \\ 
\sum_{P \in \ell \setminus \{P_0\}} \mu(P) = \mu(\ell) - \mu(P_0) & \mbox{ if } \ell\in \mathsf{L}(P_0), 
\end{cases}
\end{equation}
where $\ell$ is a projective line of $\PG(2r+1,q)$ passing through $P_0$.

\begin{lemma}\label{lemma-weight}
Let $P_0$ be a point of $\cQ_r$ and let $\mu$ be a weighted $m$-ovoid of $\cQ_r$. Then $\mu^{\downarrow}_{P_0}$ is a weighted $(m - \mu(P_0))$-ovoid of $\cQ_{r-1}$.
\end{lemma}
\begin{proof}
For every $\ell \in \mL(P_0)$, we have that
\begin{align*}
\mu^\d_{P_0} (\ell^\perp) & = \sum_{\substack{\ell_1 \in \mL(P_0), \\ \ell_1 \subset \ell^\perp}} \mu^\d_{P_0}(\ell_1) && \\
& = \mu(\ell^\perp) - \mu(P_0) && \\
& = m(q^{r-1} + 1) - q^{r-1} \mu(\ell) - \mu(P_0) && \mbox{~\big[by Lemma \ref{basic}(c)\big]} \\
& = (m - \mu(P_0)) (q^{r-1} + 1) - q^{r-1} \mu^\d_{P_0}(\ell) && \mbox{~\big[by Eq. \eqref{eq:hatmu}\big]},
\end{align*}
which shows that $\mu^\d_{P_0}$ satisfies Eq.  \eqref{eq:star}; thus, the result follows.
\end{proof}

For a weighted ovoid $\mu$ of $\cQ_r$, let 
$\|\mu\|^2$ denote the squared norm of $\mu$, i.e.
\[
\|\mu\|^2:=\sum_{P\in\PG(2r+1,q)}\mu(P)^2
=\sum_{P\in\cQ_r}\mu(P)^2,
\]
where we omit the notation for $r$, as it should be clear from the context. The next lemma relates $\|\mu\|^2$ and 
$\|\mu^\d_{P_0}\|^2$.


\begin{lemma}\label{lm:mu}
Let $\mu$ be a weighted $m$-ovoid of $\cQ_r$. Then, for any point $P_0 \in \cQ_r$, the following equality holds:
\begin{eqnarray*}
\|\mu\|^2
&=& \mu(P_0)^2+\big(\mu(P_0) + m(q-1)\big)^2 + (q+1)\cdot\sum_{P_1\in P_0^{\perp}\setminus \{P_0\}}\mu(P_1)^2 - 
\|\mu^\d_{P_0}\|^2.
\end{eqnarray*}
\end{lemma}
\begin{proof}
Let $\cE$ denote the set of pairs $(P,R)$ such that a point $P \in P_0^\perp \setminus \{P_0\}$, a point $R \notin P_0^\perp$, and $P \in R^\perp$. We will count in two ways the following quantity 
\begin{align*}
S = \sum_{(P,R)\in \cE} \mu(P) \mu(R).
\end{align*}

For a fixed point $R \in \cQ_r \setminus P_0^\perp$, applying Lemma \ref{basic}(c) to the line $\langle P_0,R\rangle$ gives \[\sum_{(P,R)\in \cE} \mu(P) = m(q^{r-1} +1) - q^{r-1} \left( \mu(R) + \mu(P_0) \right).\] 
Hence, since $\sum_{R \notin P_0^\perp} \mu(R) = q^r \left(m(q - 1) + \mu(P_0) \right)$ holds by Lemma \ref{basic}, we obtain 
\begin{align}\label{sum_eqn1}
S & = \sum_{R \notin P_0^\perp} \left(m(q^{r-1}+1)-q^{r-1}\mu(R) - q^{r-1} \mu(P_0) \right) \mu(R) \nonumber \\
& = q^{r-1}\left(\left(m(q^{r-1}+1)-q^{r-1}\mu(P_0)\right)\left(mq(q-1)+q\mu(P_0)\right)-\sum_{R \notin P_0^\perp} \mu(R)^2 \right).
\end{align}

On the other hand, for a fixed point $P \in \cQ_r \cap \left(P_0^\perp \setminus \{P_0\}\right)$, the quantity $\sum_{(P,R)\in \cE} \mu(R)$ equals $\mu(P^\perp) - \mu\left( \ell_P^\perp \right)$, where $\ell_P$ denotes the line of $\cQ_r$ joining $P_0$ and $P$. Set \begin{align*}
S_1 = \sum_{P \in P_0^\perp \setminus \{P_0\}} \mu(P) \mu(P^\perp) \; \mbox{ and } \; S_2 = \sum_{P \in P_0^\perp \setminus \{P_0\}} \mu(P) \mu\left(\ell_P^\perp\right). \quad \mbox{Then } S = S_1 - S_2.  
\end{align*}

Since $\mu(P_0^\perp \setminus \{P_0\}) = \left(m - \mu(P_0)\right)(q^r+1)$ holds by Eq. \eqref{eq:star}, we evaluate $S_1$ as follows:
\begin{align*}
S_1 & = \sum_{P \in P_0^\perp \setminus \{P_0\}} \left( m(q^r+1) -q^r \mu(P) \right) \mu(P) \\
& = m(q^r+1)^2 \left(m - \mu(P_0)\right) - q^r\cdot \sum_{P \in P_0^\perp \setminus \{P_0\}} \mu(P)^2.
\end{align*}

To evaluate $S_2$, observe that 
$\mu\left(\ell_P^\perp\right) = m(q^{r-1} + 1) - q^{r-1} \mu(\ell_P)$ by Eq. \eqref{eq:star}. Therefore, 
\begin{align*}
S_2 & = \sum_{P \in P_0^\perp \setminus \{P_0\}} \mu(P)\left( m(q^{r-1} + 1) - q^{r-1} \mu(\ell_P) \right)  \\
& = m(q^{r-1}+1) \left( m - \mu(P_0) \right)(q^r+1) - q^{r-1}\cdot \sum_{P \in P_0^\perp \setminus \{P_0\}} \mu(P) \mu(\ell_P).
\end{align*}
Further, 
\begin{align*}
\sum_{P \in P_0^\perp \setminus \{P_0\}} \mu(P) \mu(\ell_P) & =  \sum_{\ell \in \mL(P_0)} \left( \mu(\ell) - \mu(P_0) \right) \mu(\ell) \\
& = \sum_{\ell \in \mL(P_0)} \left(\mu(P_0) \left( \mu(\ell) - \mu(P_0) \right) + \mu_{P_0}^\d(\ell)^2\right) \\
& = \mu(P_0) \left( \mu\left( P_0^\perp \right)  - \mu(P_0) \right) + \sum_{\ell \in \mL(P_0)}\mu_{P_0}^\d(\ell)^2 \\
& = \mu(P_0) \left(m - \mu(P_0)\right) (q^r+1) + \|\mu_{P_0}^\d\|^2.
\end{align*}
Thus, we finally obtain
\begin{align}
  S_2 & = m(q^{r-1}+1) \left( m - \mu(P_0) \right)(q^r+1) - q^{r-1} \left( \mu(P_0) \left(m - \mu(P_0)\right) (q^r+1) + \|\mu_{P_0}^\d\|^2 \right) \nonumber\\
  \intertext{so}
S &= S_1 - S_2 \nonumber \\
 &=  m (m - \mu(P_0)) q^{r-1} (q^r+1) (q-1) + \mu(P_0) (m - \mu(P_0)) q^{r-1} (q^r+1) \nonumber \\
 & - q^r\cdot \sum_{P \in P_0^\perp \setminus \{P_0\}} \mu(P)^2 + q^{r-1}\cdot \|\mu_{P_0}^\d\|^2.\label{sum_eqn2}
\end{align}

Equating Eqs. \eqref{sum_eqn1} and \eqref{sum_eqn2} and simplifying the result completes the proof of the lemma.
\end{proof}

\begin{corollary}\label{cor:lines_modular}
For $r\ge 1$ and every weighted $m$-ovoid $\mu$ of $\cQ_r$, one has
\begin{equation*}
\|\mu\|^2
\equiv
\begin{cases}
-2qm^2+(q+1)(q^{r}+1)m  & \mbox{ if } r \mbox{ is even }\\
(q^2+1) m^2 & \mbox{ if } r \mbox{ is odd }
\end{cases}
\mod{2(q+1)}.
\end{equation*}
\end{corollary}
\begin{proof}
In this proof $\equiv$ stands for equivalence modulo $2(q+1)$. We prove the assertion by induction on $r$. For $r=1$, we have $\mu(P)=m$ for each of the $q^2+1$ points of $\cQ_1$ and the claim follows.

Now suppose that $r\ge 2$, let $P_0$ be any point of $\cQ_r$ and put $x:=\mu(P_0)$. For each integer $t$ we have $(q+1)t^2\equiv (q+1)t$ and hence, by Eq. \eqref{eq:star}, 
\begin{align*}
 \sum_{P\in P_0^\perp\setminus\{P_0\}}(q+1)\mu(P)^2
 & \equiv
 \sum_{P\in P_0^\perp\setminus\{P_0\}}(q+1)\mu(P)
   \equiv (q+1)(q^{r}+1)\big(m-x\big).
\end{align*}
Lemma \ref{lm:mu} shows thus
\begin{align*}
\sum_{P\in\cQ_r} \mu(P)^2 & \equiv x^2+\big(x + m(q-1)\big)^2  
+ (q+1)(q^r+1)\big(m-x\big)- \sum_{\ell\in \mathsf{L}(P_0)}\mu^{\downarrow}_{P_0}(\ell)^2.
\end{align*}
Now we apply the induction hypothesis to the quadric $P_0^\perp/P_0$ (with point-set $\mathsf{L}(P_0)$) induced by $\cQ_r$ and the weighted $(m-x)$-ovoid $\mu^{\downarrow}_{P_0}$ of $\cQ_{r-1}$. When $r$ is even, this gives
\begin{eqnarray*}
\sum_{P\in\cQ_r} \mu(P)^2 & \equiv & x^2+\big(x + m(q-1)\big)^2 + (q+1)(q^r+1)\big(m-x\big)- (q^2+1)(m-x)^2
 \\
 &\equiv& (1-q^2)x^2+2xmq(q+1)-x(q+1)(q^r+1)-2qm^2+(q+1)(q^r+1)m
 \\
 &\equiv& -2qm^2+(q+1)(q^r+1)m
\end{eqnarray*}
where we use $(1-q^2)x^2\equiv (1-q^2)x$ in the last step.
When $r$ is odd, we find instead
\begin{eqnarray*}
\sum_{P\in\cQ_r} \mu(P)^2 & \equiv& x^2+\big(x + m(q-1)\big)^2 + (q+1)(q^r+1)\big(m-x\big)\\
 && +2q(m-x)^2-(q+1)(q^{r-1}+1)(m-x)
 \\
 &\equiv&  x^2+\big(x + m(q-1)\big)^2 +2q(m-x)^2 \\
 &\equiv&  2(q+1)x^2-2xm(q+1)+m^2(q^2+1)\equiv m^2(q^2+1)
\end{eqnarray*}
as desired.
\end{proof}

\section{A modular equality for $m$}\label{sect:eq}
In this section, we prove Theorem \ref{th:1}. Let $\cO$ be an $m$-ovoid of $\cQ_r$, $r\geq 2$, and fix a point $P_0\in \cQ_r$. Recall that the $(0,1)$-characteristic function $\chi$ of $\mathcal{O}$ is a weighted $m$-ovoid of $\cQ_r$, and $\chi^{\downarrow}_{P_0}$ is a weighted $(m - \chi(P_0))$-ovoid of $\cQ_{r-1}$ by Lemma \ref{lemma-weight}.


\begin{lemma}\label{lines}
The following holds:
\begin{equation}\label{eq:doublecount}
\|\chi^{\downarrow}_{P_0}\|^2
= \chi(P_0) + \left( \chi(P_0) + m (q-1) \right)^2 - \chi(P_0) (q+1) (q^r+1) + m (q^r+q).
\end{equation}
\end{lemma}
\begin{proof}
The result follows from Lemma \ref{lm:mu} applied to $\chi$ in the role of $\mu$ (observe that $\left(\chi(P)\right)^2=\chi(P)$ for any point $P$).
\end{proof}

The following lemma immediately follows from Corollary \ref{cor:lines_modular} applied to $\chi^{\downarrow}_{P_0}$.

\begin{lemma}\label{lines_modular}
Let $\equiv$ denote equivalence modulo $2(q+1)$.
The following holds:
\begin{equation}\label{eq:modcount}
\|\chi^{\downarrow}_{P_0}\|^2
\equiv 
\begin{cases}
-2q \left(m-\chi(P_0)\right)^2 + (q+1)(q^{r-1} + 1)\left(m-\chi(P_0)\right)  & \mbox{ if } r \mbox{ is odd,}\\ 
(q^2+1) \left(m-\chi(P_0)\right)^2 & \mbox{ if } r \mbox{ is even.}
\end{cases}
\end{equation}
\end{lemma}

We are now in a position to prove our main result, 
Theorem \ref{th:1}.

\begin{proof}
Fix a point $P_0 \in \cQ_r \setminus \cO$. By Lemmas \ref{lines} and  \ref{lines_modular}, we have two equalities for $\|\chi^{\downarrow}_{P_0}\|^2$.

Suppose that $r$ is odd. 
Equating (modulo $2(q+1)$) Eqs. \eqref{eq:doublecount} and \eqref{eq:modcount} gives 
\[
(q^2+1) m^2 - (q^{r - 1} + 1) m \equiv 0 \mod{2(q+1)}, 
\]
which is equivalent to either $\big(2(m^2 - m) \equiv 0 \mod{2(q+1)}\big)$ or $\big((q+3) (m^2 - m) \equiv 0 \mod{2(q+1)}\big)$, according as $q$ is odd or even, respectively. In the latter case, note that $(q+3) (m^2 - m) \equiv 2(m^2 - m) \mod{2(q+1)}$; hence $2(m^2 - m) \equiv 0 \mod{2(q+1)}$ holds in the even characteristic case as well.

Similarly, if $r$ is even, we obtain 
$$
2 m^2 + (q^{r} + q) m \equiv 0 \mod{2(q+1)}, 
$$
and the result follows.    
\end{proof}

We now determine the number of solutions of Eq. \eqref{eq:main}. 

\begin{lemma}\label{lm:val}
Let $q+1 = p_1^{k_1} \cdot \ldots \cdot p_t^{k_t}$ be the prime factorization of $q+1$. Then the following hold.
\begin{enumerate}
\renewcommand{\labelenumi}{\rm (\alph{enumi})}
\item There are $2^t$ integers $m$, with $0 \le m \le q$, such that $m^2 - m \equiv 0 \pmod{q+1}$.
\item If $q$ is even, there are $p_1^{\lfloor k_1/2 \rfloor} \cdot \ldots \cdot p_t^{\lfloor k_t/2 \rfloor}$ integers $m$, with $0 \le m \le q$, such that $m^2 \equiv 0 \pmod{q+1}$. 
\item If $q \equiv -1 \mod{4}$, then there are $p_1^{\lfloor k_1/2 \rfloor} \cdot \ldots \cdot p_t^{\lfloor k_t/2 \rfloor}$ integers $m$, with $0 \le m \le q$, such that $m^2 + \frac{q+1}{2} m \equiv 0 \pmod{q+1}$. 
\item If $q \equiv 1 \mod{4}$, then there are $2 \cdot p_2^{\lfloor k_2/2 \rfloor} \cdot \ldots \cdot p_t^{\lfloor k_t/2 \rfloor}$ integers $m$, with $0 \le m \le q$, such that $m^2 + \frac{q+1}{2} m \equiv 0 \pmod{q+1}$. 
\end{enumerate}
\end{lemma}
\begin{proof}
Let $f(m) = m^2 + a m$, for some integer $a$. From the Chinese Remainder Theorem we have that 
\begin{equation} \label{cong}
f(m) \equiv 0 \pmod{q+1}
\end{equation}
has a solution if and only if each of the equations 
\begin{equation} \label{cong_1}
f(m) \equiv 0 \pmod{p_i^{k_i}}, \;\; 1 \le i \le t, 
\end{equation}
has a solution. Moreover, if Eq. \eqref{cong_1} has $n_i$ solutions, then Eq. \eqref{cong} has $n_1 \cdot \ldots \cdot n_t$ solutions. Since $m^2 - m \equiv 0 \pmod{p_i^{k_i}}$ has $2$ solutions and $m^2 \equiv 0 \pmod{p_i^{k_i}}$ admits $p_i^{\lfloor k_i/2 \rfloor}$ solutions, $1 \le i \le t$, statements (a) and (b) follow.

If $q$ is odd, then, assuming that $p_1=2$, one has  $\frac{q+1}{2} \equiv 0 \pmod{p_i^{k_i}}$, $2 \le i \le t$, and $\frac{q+1}{2} \equiv 2^{k_1 - 1} \pmod{2^{k_1}}$ or $\frac{q+1}{2} \equiv 1 \pmod{2}$, according as $q \equiv -1 \pmod{4}$ or $q \equiv 1 \pmod{4}$. The fact that $m^2 + m \equiv 0 \pmod{2}$ has $2$ solutions and that $m^2 + 2^{k_1 - 1} m \equiv 0 \pmod{2^{k_1}}$ admits $2^{\lfloor k_1/2 \rfloor}$ solutions, shows (c) and (d).
\end{proof}

Combining Lemma \ref{lm:val} together with Theorem \ref{th:1}, we get that the number of admissible values of $m$ for an $m$-ovoid of $\cQ_r$ equals:
\begin{eqnarray*}
 2^t \cdot (q^{r-2} + q^{r - 4} + \dots + q) + 1, & \mbox{ if } r \mbox{ is odd, } \\
 p_1^{\lfloor k_1/2 \rfloor} \cdot \ldots \cdot p_t^{\lfloor k_t/2 \rfloor} \cdot (q^{r - 2} + q^{r - 4} + \dots + q^2 + 1), & \mbox{ if } r \mbox{ is even and } q \not\equiv 1 \pmod{4}, \\
 2 \cdot p_2^{\lfloor k_2/2 \rfloor} \cdot \ldots \cdot p_t^{\lfloor k_t/2 \rfloor} \cdot (q^{r - 2} + q^{r - 4} + \dots + q^2 + 1), & \mbox{ if } r \mbox{ is even and } q \equiv 1 \pmod{4}.
\end{eqnarray*}

\section{A lower bound for $m$ and some characterization results}\label{sect:ineq}

In this section we slightly improve on the lower bound $m \ge (\sqrt{4 q^{r+1} + 9} - 3)/(2q-2)$ for an $m$-ovoid of $\cQ^-(2r+1, q)$, which was shown in \cite[Theorem 13]{BKLP}.

In particular, an $m$-ovoid of $\cQ^-(7,q)$ may exist only when $m\ge q+1$. All known examples of $(q+1)$-ovoids of $\cQ^-(7, q)$ arise by applying the field reduction to the points of $\cQ^-(3, q^2)$, see \cite{K}. In this case, the $(q+1)$-ovoid consists of the points of $q^4+1$ pairwise disjoint lines $\ell_1, \dots, \ell_{q^4+1}$ forming a {\em $1$-system}. 
Recall that a {\bf $1$-system} in $\cQ^-(7,q)$ is a set of $q^4+1$ pairwise disjoint lines $\ell_1, \dots, \ell_{q^4+1}$ such that every plane of $\cQ^-(7, q)$ containing $\ell_i$ is disjoint from $\cup_{j = 1, j \ne i}^{q^4+1} \ell_j$. 
The $1$-systems in $\cQ^-(7, q)$ are unique  \cite{LT1, LT2}. Conversely, we will show that a $(q+1)$-ovoid of $\cQ^-(7, q)$, $q\in\{2,3\}$, consists of the points covered by the lines of a $1$-system of $\cQ^-(7, q)$.


\begin{theorem}\label{theorembound}
If $\cQ^-(2r+1,q)$ possesses an $m$-ovoid with $m>0$, then
\[
m\geq \frac{\sqrt{4q^{r+1}+4q+1}-3}{2(q-1)}. 
\]
\end{theorem}
\begin{proof}
Let $\cO$ be an $m$-ovoid of $\cQ_r$, $\chi$ the characteristic function of $\cO$. Fix a point $P_0 \in \cO$, and for every line $\ell\in \mathsf{L}(P_0)$, define $t_{\ell}:=\chi^{\downarrow}_{P_0}(\ell)=|\ell\cap \cO| - \chi(P_0)$. Then $|P_0^\perp \cap \cO| = m (q^r+1) - q^r$ holds by Eq. \eqref{eq:ovoid-point}, and hence 
\begin{align}
\sum_{\ell \in \mL(P_0)} t_\ell & = |P_0^\perp\cap {\cal O}| -1 = (m-1)(q^r+1). \label{eq:minus1}
\end{align}
Moreover, by Lemma \ref{lines}, we have
\begin{align}
\sum_{\ell \in \mL(P_0)} t_\ell^2 & = 1 + \left(1 + m(q-1)\right)^2 - (q+1)(q^r+1) + m (q^r+q). \label{eq:minus2}
\end{align}
Therefore, subtracting Eq. \eqref{eq:minus1} from Eq. \eqref{eq:minus2}, we obtain
\begin{align}
\sum_{\ell \in \mL(P_0)} t_\ell (t_\ell - 1) & = m^2(q-1)^2 + 3m(q-1) - q^{r+1} - q + 2. \label{eq:minus}
\end{align}
The left-hand side of Eq. \eqref{eq:minus} is non-negative, whereas its right-hand side is a quadratic polynomial in $m$ with positive leading coefficient, whose largest root is 
\[
m_1=\frac{\sqrt{4q^{r+1}+4q+1}-3}{2(q-1)}.
\]
Hence $m\ge m_1$, which completes the proof.
\end{proof}

For $r=3$, we find that $m\ge q+1$. If $m=q+1$ and $P_0\in\cal O$, then Eq. \eqref{eq:minus} reads as
\begin{align}\label{eq:minus3}
\sum_{\ell \in \mL(P_0)} t_\ell (t_\ell - 1) & = q(q-1). 
\end{align}
It is readily seen by Eq. \eqref{eq:minus3} that if a line on $P_0$ is contained in $\cal O$, then every other line of $\cQ^-(7, q)$ through $P_0$ meets $\cal O$ in one or two points. Suppose that this occurs for every point of $\cal O$. Then it turns out that there are exactly $q^4+1$ lines contained in $\cal O$ and that the set $\cal L$ of these lines forms a partition of the $(q+1)$-ovoid. Furthermore, in this case a plane of $\cQ^-(7, q)$ that contains a line of $\cal L$ does not meet any other line of $\cal L$, since it intersects $\cO$ in exactly $q+1$ points. Therefore $\cal L$ is a $1$-system of $\cQ^-(7, q)$. In the next two theorems, we will show that this is the case when 
$q\in\{2,3\}$.

\begin{theorem}
The elliptic quadric $\cQ^-(7, 2)$ has a non-trivial $m$-ovoid only for $m=3$ and $m=4$. Moreover, every $3$-ovoid of $\cQ^-(7,2)$ is the union of the lines of a $1$-system of $\cQ^-(7, 2)$. 
\end{theorem}
\begin{proof}
For an $m$-ovoid of $\cQ^-(7, 2)$ we have that $3$ divides $m(m-1)$ by Theorem \ref{th:1}. It follows that a non-trivial $m$-ovoid can exist only for $m\in\{1,3,4,6\}$. The case $m = 1$ does not occur by Theorem~\ref{theorembound}. Hence $m \ne 6$, since the complement of a $6$-ovoid is a $1$-ovoid. Let $\cO$ be a $3$-ovoid of $\cQ^-(7, 2)$ and let $P_0 \in \cO$. Note that if $\ell \in \mL(P_0)$, then $t_{\ell} = |\ell \cap \cO| - 1 \in \{0, 1, 2\}$ and $\sum_{\ell \in \mL(P_0)} t_\ell(t_\ell-1) = 2$ by Eq. \eqref{eq:minus3}. Therefore,  every point of $\cO$ lies on exactly one line contained in $\cO$ and hence the $3$-ovoid arises from a $1$-system of $\cQ^-(7, 2)$. (Note that a $4$-ovoid is the complement of a 3-ovoid.)
\end{proof}

\begin{theorem}\label{th:qminus73}
A non-trivial $m$-ovoid of $\cQ^-(7, 3)$ can exist only for $m\in\{4, 5, 8, 9\}$. Moreover, every $4$-ovoid of $\cQ^-(7, 3)$ is the union of the lines of a $1$-system of $\cQ^-(7, 3)$. 
\end{theorem}
\begin{proof}
Theorem \ref{th:1} and Theorem \ref{theorembound} imply that a non-trivial $m$-ovoid of $\cQ^-(7,3)$ can exist only for $m\in\{4,5,8,9\}$. Let $\cO$ be a $4$-ovoid of $\cQ^-(7, 3)$ and let $P_0 \in {\cal O}$. If $\ell \in \mL(P_0)$, then $t_{\ell} = |\ell \cap \cO| - 1 \in \{0, 1, 2, 3\}$ and $\sum_{\ell \in \mL(P_0)} t_\ell(t_\ell-1) = 6$ by Eq.  \eqref{eq:minus3}. It follows that two possibilities arise: either there is exactly one line of $\cQ^-(7, 3)$ through $P_0$ contained in $\cO$ and the remaining lines of $\mL(P_0)$ have at most two points in common with $\cO$ or there are exactly three lines of $\cQ^-(7, 3)$ through $P_0$ intersecting $\cO$ in three points and the remaining lines of $\mL(P_0)$ have at most two points in common with $\cO$. If the latter case does not occur, then every point of $\cO$ lies on exactly one line contained in $\cO$ and hence the $4$-ovoid arises from a $1$-system of $\cQ^-(7, 3)$. 

Assume by way of contradiction that the second case occurs for some point $P_0$ of $\cal O$ and let $\ell_1,\ell_2,\ell_3$ be the three lines of $\cQ^-(7, 3)$ on $P_0$ such that $|\ell_i\cap{\cal O}|=3$, i.e., $t_{\ell_i}= 2$. We will evaluate in two different ways the number of lines $\ell$ belonging to $\mL(P_0)$ and such that $t_\ell = 0$, i.e., meeting $\cO$ exactly in $P_0$. Since the point $P_0$ lies on $112$ lines of $\cQ^-(7, 3)$ and $\sum_{\ell \in \mL(P_0)} t_\ell = 84$ by Eq.  \eqref{eq:minus1}, it follows that there are exactly 
\[
112 - (84-6) -3 = 31
\] 
lines of $\mL(P_0)$ meeting $\cO$ exactly in $P_0$.

On the other hand, each of the $10$ planes on $\ell_i$ meets $\cal O$ in $4$ points and hence in each of these planes there are precisely two lines through $P_0$ that meet $\cal O$ only in $P_0$. This shows that $P_0$ lies on $20$ lines of the quadric that lie in $\ell_i^\perp$ and meet $\cal O$ only in $P_0$. Moreover the plane $\sigma_{i, j} = \langle \ell_i,\ell_j\rangle$, $1 \le i < j \le 3$ meets $\cQ^-(7, 3)$ only in $\ell_i \cup \ell_j$ and the 4-space $\sigma_{i, j}^\perp$ meets $\cQ^-(7, 3)$ in a cone having as vertex the point $P_0$  and as base a $\cQ^-(3, 3)$. Since $|\sigma_{i, j} \cap \cO| = 5$, by Lemma \ref{basic}(c) we have that $|\sigma_{i, j}^\perp \cap \cO|  = 1$, and hence $\sigma_{i, j}^\perp$ meets ${\cal O}$ only in $P_0$. Similarly, the solid $\Pi = \langle \ell_1,\ell_2,\ell_3 \rangle$ as well as the solid $\Pi^\perp$ intersect $\cQ^-(7, 3)$ in a quadratic cone having as vertex $P_0$; since $|\Pi \cap \cO| \ge 7$, Lemma~\ref{basic}(c) gives $|\Pi^\perp \cap \cO|  \le 1$, and hence $\Pi^\perp$ meets ${\cal O}$ only in $P_0$. Let $a_I$ denote the number of lines of $\cQ^-(7, 3)$ on $P_0$ that meet $\cal O$ only in $P_0$ and that are contained in $\ell_i^\perp$ for all $i\in I$. Thus, 
\begin{align*}
& a_{\{i\}} = 20 && \mbox{ for } i=1,2,3, \\
& a_{\{i,j\}} = 10 && \mbox{ for } 1 \le i< j \le 3, \\ 
& a_{\{1,2,3\}} = 4. 
\end{align*}
Hence, $P_0$ lies on at least
\[
3\cdot 20-3\cdot 10+4=34
\]
lines of $\cQ^-(7, 3)$ that meet ${\cal O}$ only in $P_0$, a contradiction.
\end{proof}

A natural question arising from Theorem \ref{th:qminus73} concerns the existence of a $5$-ovoid of $\cQ^-(7, 3)$. Note that the complement of a $5$-ovoid is an $8$-ovoid; hence one could ask whether or not it is possible to obtain an $8$-ovoid of $\cQ^-(7, 3)$ by glueing together two disjoint $4$-ovoids of $\cQ^-(7, 3)$. Some computations performed with Magma \cite{magma} show that the sets of points covered by two distinct $1$-systems of a $\cQ^-(7, 3)$ always have in common at least $4$ points. 


\Acknowledgements

\end{document}